\documentclass[12pt]{amsart}

\usepackage[english]{babel}

\usepackage[letterpaper,top=2cm,bottom=2cm,left=3cm,right=3cm,marginparwidth=1.75cm]{geometry}
\usepackage{amsmath}
\usepackage{graphicx}
\usepackage[colorlinks=true, allcolors=blue]{hyperref}
\usepackage{amsthm}
\usepackage{amssymb}
\usepackage{amsfonts,amscd}
\usepackage{fullpage}

\usepackage{tikz}
\usetikzlibrary{positioning, automata, arrows, shapes, matrix, arrows, decorations.pathmorphing, decorations.pathreplacing}
\tikzset{main node/.style={circle,draw,minimum size=0.3em,inner sep=0.5pt}}
\tikzset{small node/.style={circle, draw,minimum size=0.1cm,scale=0.3,fill}}
\tikzset{anchorbase/.style={>=To,baseline={([yshift=-0.5ex]current bounding box.center)}}}
\tikzset{->-/.style={decoration={markings, mark=at position 0.53 with {\arrow[-Latex]{>}}},postaction={decorate}}}
\usepackage{xspace}
\usepackage{adjustbox}
\theoremstyle{plain}
\newtheorem{theorem}{Theorem}[section]
\newtheorem{lemma}[theorem]{Lemma}
\newtheorem{prop}[theorem]{Proposition}

\theoremstyle{definition}
\newtheorem{defn}[theorem]{Definition}
\newtheorem{exa}[theorem]{Example}

\theoremstyle{remark}
\newtheorem{rmk}[theorem]{Remark}
\numberwithin{equation}{section}
\newcommand{\al}{\alpha}
\newcommand{\be}{\beta}

\newcommand{\Q}{{\mathbb Q}}
\newcommand{\R}{{\mathbb R}}
\newcommand{\Z}{{\mathbb Z}}
\newcommand{\C}{{\mathbb C}}
\newcommand{\xr}{x_r}
\newcommand{\xs}{x_s}
\newcommand{\ra}{\rightarrow}

\newcommand{\vnk}{V_{n,k}}
\newcommand{\vcnk}[2]{V_{#1,#2}}
\newcommand{\vtnk}[3]{V_{#1,#2,#3}}
\newcommand{\cnk}{{\mathcal C}_{n,k}}
\newcommand{\bnk}{{\mathcal B}_{n,k}}
\newcommand{\mnk}{{\mathcal M}_{n,k}}
\newcommand{\bcnk}[2]{{\mathcal B}_{#1,#2}}
\newcommand{\mcnk}[2]{{\mathcal M}_{#1,#2}}
\newcommand{\btnk}[3]{{\mathcal B}_{#1,#2,#3}}
\newcommand{\latnk}{{\Lambda_{n,k}}}
\newcommand{\thelab}{\lambda}
\newcommand{\pnkj}{{\Phi(n,k,j)}}
\newcommand{\psnkj}{{\Psi(n,k,j)}}
\newcommand{\Span}{\text{\rm Span}}

\makeatletter
\newcommand{\exterior}[1]{\mathop{\mathpalette\exterior@{#1}}}
\newcommand{\exterior@}[2]{%
  \raisebox{\depth}{%
  \fontsize{\sf@size}{0}%
  \m@th
  $\ifx#1\displaystyle\textstyle\else#1\fi\bigwedge$}%
  ^{\mspace{-2mu}#2}%
  \kern-\scriptspace
}
\makeatother

\title{Positivity properties for spherical functions of maximal Young subgroups}
\author{R. M. Green}
\address{Department of Mathematics\\
University of Colorado Boulder, Campus Box 395\\
Boulder, Colorado\\
USA, 80309}
\email{rmg@colorado.edu}



\begin{document}
\maketitle

\begin{abstract}
Let $S_k \times S_{n-k}$ be a maximal Young subgroup of the symmetric group $S_n$.
We introduce a basis $\bnk$ for the coset space $S_n/S_k \times S_{n-k}$ that is naturally parametrized by 
the set of standard Young tableaux with $n$ boxes, at most two rows, and at most $k$ boxes in the second row.
The basis $\bnk$ has positivity properties that resemble those of a root system, and 
there is a composition series of the coset space in which each term is spanned by the basis elements that it contains. 
We prove that the spherical functions of the associated Gelfand pair are nonnegative linear combinations of the $\bnk$.
\end{abstract}

\vskip 20pt
\centerline{\bf Published in the Annals of Combinatorics}
\vskip 20pt

\section{Introduction}

A pair $(G, K)$ of finite groups is called a {\it Gelfand pair} if $K \leq G$ and the permutation representation of $G$ 
on the set of left cosets $X = G/K$ of $K$ in $G$ is multiplicity-free as a $\C G$-module. One source of examples of Gelfand
pairs arises from the action of a group $G$ on a finite metric space $(X, d)$. Such an action is said to be
{\it distance transitive} if for all $(x_1, y_1), (x_2, y_2) \in X \times X$, we have $d(x_1, y_1) = d(x_2, y_2)$ if 
and only if there exists a $g \in G$ satisfying $g(x_1) = x_2$ and $g(y_1) = y_2$; this condition implies that $G$ acts 
transitively as a group of isometries of $X$. If $K$ is the stabilizer of $x_0 \in X$ under a distance transitive action 
of $G$, then $(G, K)$ is a Gelfand pair \cite[Lemma 4.3.4, Example 4.3.7]{ceccherini08}. Furthermore, in this case, the 
number of orbits of $K$ on $X$ (i.e., the rank of $G$ acting on $X$ as a permutation group as in \cite[Definition 8.2.4]{green13}) 
is equal to the number of irreducible direct summands of the permutation module on the cosets $G/K$ 
\cite[Corollary 4.4.3 (iii)]{ceccherini08}.

The action of a Weyl group on the weights of a minuscule representation of a simple Lie algebra satisfies the conditions of
the previous paragraph with respect to Euclidean distance by \cite[Theorem 8.2.22 (ii)]{green13}, so each minuscule representation 
of a simple Lie algebra gives rise to a Gelfand pair. In this paper, we concentrate on the special case where the Lie algebra has 
type $A_{n-1}$, which means that the Gelfand pair $(G, K)$ is given by $(S_n, S_k \times S_{n-k})$ for some $0 < k < n$. We 
will assume that $n \geq 2$ throughout, and without loss of generality that $k \leq n/2$.
In this case, each left coset of $X = G/K$ can be naturally identified with a squarefree monomial of degree $k$ in the commuting 
indeterminates $x_1, x_2, \ldots, x_n$, where the action of $G$ is the natural action on subscripts, and where the identity coset 
$K$ is identified with the monomial $x_1x_2 \cdots x_k$.
If we denote the set of $\C$-valued functions on $X$ by $L(X)$, then $L(X)$ decomposes as a $\C G$-module into
a direct sum of pairwise nonisomorphic irreducible representations $$
L(X) \cong V_0 \oplus V_1 \oplus \cdots \oplus V_k
.$$  We will identify the vector space $L(X)$ with the linear span, $\vnk$, of
the squarefree monomials of degree $k$ in the commuting indeterminates $\{x_1, x_2, \ldots, x_n\}$. We will refer to both 
these versions of the coset basis as the {\it monomial basis}, and denote it by $\mnk$.

It follows from Frobenius reciprocity that each of the $V_j$ has a $1$-dimensional $K$-invariant submodule.
For each $0 \leq j \leq k$, the $j$-th spherical function $\pnkj \in L(X)$ is defined to be the element of this 
$1$-dimensional submodule that is normalized so that $\pnkj$ sends the identity coset $K=x_1x_2\cdots x_k$ to $1$. 
The value of the spherical function $\pnkj$ on a coset $gK$ turns out to be a function of the distance $d$ between $gK$ 
and $K$ in the natural metric on $X$. These spherical functions
are known explicitly \cite[Theorem 6.1.10]{ceccherini08} and are sometimes called dual Hahn polynomials. They have
applications to random walks and the Bernoulli--Laplace diffusion model \cite[\S3]{diaconis87}. We will not use the metric 
in this paper, and instead view the spherical functions $\pnkj$ as homogeneous polynomials 
of degree $k$. Because the irreducible representations of $S_n$ over $\C$ are defined over $\Q$, we will work over the field 
$\Q$ unless stated otherwise, but scalars can be extended if necessary.

If $B$ is a basis for an $F$-vector space $V$ with $F \leq \R$, we say that an element $v = \sum_{b \in B} \lambda_b b$ is 
{\it $B$-positive with coefficients $\{\lambda_b\}_{b \in B}$} if we have $\lambda_b \geq 0$ for all $b \in B$. The spherical 
functions $\pnkj$ are generally not $\mnk$-positive elements of $\vnk$, but in this paper, we will introduce a basis $\bnk$ for 
$\vnk$ with respect to which the spherical functions are $\bnk$-positive. To illustrate this, consider the case $n=4$ and $k=1$, 
where we have $\mcnk{4}{1}= \{x_1, x_2, x_3, x_4\}$ and $$
\bcnk{4}{1} = \{x_1 - x_2, x_2 - x_3, x_3 - x_4, x_3 + x_4\}
.$$ If we identify the $x_i$ with an orthonormal basis of $\R^n$, the latter basis corresponds to the basis of simple roots of 
type $D_4$, as in \cite[\S2.10]{humphreys90}. In this case, $\Phi(4,1,0)$ is both
$\mcnk{4}{1}$-positive and $\bcnk{4}{1}$-positive: $$
\Phi(4,1,0) =
x_1 + x_2 + x_3 + x_4 = (x_1 - x_2) 
+ 2(x_2 - x_3)
+ (x_3 - x_4)
+ 2(x_3 + x_4)
,$$ whereas $\Phi(4,1,1)$ is not $\mcnk{4}{1}$-positive, but is $\bcnk{4}{1}$-positive: $$
\Phi(4,1,1) = x_1
- \frac{1}{3} x_2
- \frac{1}{3} x_3
- \frac{1}{3} x_4
= (x_1 - x_2) 
+ \frac{2}{3} (x_2 - x_3)
+ \frac{1}{3} (x_3 - x_4)
.$$

The purpose of this paper is to study $\bnk$-positivity for arbitrary $n \geq 2$ and $0 < k \leq n/2$. We replace the 
root system of type $D_n$ by a generalization called $k$-roots. Using a suitable total order, we can construct 
a canonical basis $\bnk$ of $k$-roots analogous to the simple roots in the case $k=1$. The basis $\bnk$ is 
naturally parametrized by the set of lattice words in the alphabet $\{1, 2\}$ that have length $n$ and at most $k$ occurrences of 
$2$, or equivalently (see Remark \ref{rmk:tableaux}) by the set of standard Young tableaux with $n$ boxes that have at most two 
rows and at most $k$ boxes in the second row. The basis $\bnk$ may be constructed in other ways, for example by using 
Kazhdan--Lusztig theory (see Remark \ref{rmk:kl}), but the $k$-root approach has the advantage that it is easy to deal with 
computationally.

The results of this paper are largely self-contained, although the key result Proposition \ref{prop:reduction} is implicit in 
recent work of the author and T. Xu \cite{gx3} on the case $k=2$ in a much more general setting. In \cite{gx3}, a 
$k$-root is defined to be a symmetrized tensor
product of $k$ mutually orthogonal roots in the sense of Lie theory. The cases we study in this paper correspond to performing
this construction with a root system of type $D$, where convenient Euclidean coordinates are available. We therefore usually
dispense with the root system point of view, and instead think of $k$-roots as polynomials in these Euclidean coordinates.
It should be noted that the polynomials corresponding to certain pairs of orthogonal roots, such as $(x_1 - x_2)(x_1 + x_2)$, 
do not appear in the construction because they are not linear combinations of squarefree monomials, and such polynomials are not
counted as $k$-roots for the purposes of this paper. 

We develop the combinatorial tools needed to define and study the canonical basis $\bnk$ in sections \ref{sec:kroots} and 
\ref{sec:canbas}. Although the initial definition of $\bnk$ in Definition \ref{def:defects} may 
seem ad hoc, we will prove in Theorem \ref{thm:canbas} that $\bnk$ has a simple characterization as the set of positive 
$k$-roots that are minimal in the sense of being indecomposable into sums of other positive $k$-roots.
Section \ref{sec:main} explores some applications of $k$-roots to representation theory. Theorem \ref{thm:main}
gives a closed formula for the spherical functions in terms of $k$-roots; the formula does not involve the metric on 
the coset space $G/K$. We use this formula to prove that the main result of this paper, which is that the spherical functions 
are $\bnk$-positive. Theorem \ref{thm:revlat} 
gives a sufficient condition for a monomial basis element to be $\bnk$-positive. 

\section{k-roots}\label{sec:kroots}

Let $n \geq 2$ and $0 < k \leq n/2$ be integers, and let $\vnk$ be the $\Q$-vector space of dimension $\binom{n}{k}$ with basis 
consisting of the set $\mnk$ of all squarefree monomials of degree $k$ in the commuting indeterminates $\{x_1, x_2, \ldots, x_n\}$.

The set $\mnk$ can be ordered lexicographically, as follows. Let $I = \{i_1, i_2, \ldots, i_k\}$ and 
$J = \{j_1, j_2, \ldots, j_k\}$ be two distinct subsets of $\{1, 2, \ldots, n\}$ of size $k$, and let
$x_I = x_{i_1} x_{i_2} \cdots x_{i_k}$ and $x_J = x_{j_1} x_{j_2} \cdots x_{j_k}$ be the corresponding squarefree monomials.
If $t$ is the smallest element of the symmetric difference $I \ \Delta \ J$, then we say $x_I \prec x_J$ if $t \in I$, 
and $x_J \prec x_I$ if $t \in J$. (For example, we have $x_1 x_2 x_5 \prec x_1 x_3 x_4$, with $t=2$.)

We order the $\Q$-vector space $V_{n, k}$ by saying that a nonzero vector $v \in V_{n, k}$ satisfies $v > 0$ (respectively, $v < 0$)
if the lexicographically minimal monomial appearing in $v$ has positive (respectively, negative) coefficient. We then say that $v_1 < v_2$
if $v_2 - v_1 > 0$. This makes $V_{n, k}$ into a totally ordered vector space. (Note that we have $x_1 x_2 x_5 > x_1 x_3 x_4$ in this case.)

\begin{defn}\label{def:cnk}
Let $\cnk$ be the subset of $\vnk$ consisting of all elements of the form $$
\prod_{r = 1}^k (\pm x_{i_{2r-1}} \pm x_{i_{2r}})
,$$ where the signs are chosen independently and where the set $\{i_1, i_2, \ldots, i_{2k}\}$ is a set of $2k$ distinct indices 
from the set $\{1, 2, \ldots, n\}$. An element of $\cnk$ is called a {\it $k$-root}.
A $k$-root is called {\it positive} if it is positive in the ordering on $\vnk$, and {\it negative}
otherwise.
\end{defn}  

\begin{rmk}\label{rmk:scalars}
The definitions imply that if $\lambda$ is a scalar and $\al$ is a $k$-root, then $\lambda \al$ is also a $k$-root if and only if 
$\lambda = \pm 1$. Note that for each $k$-root $\alpha$, the factors $(\pm x_i \pm x_j)$ are well-defined up to order 
and multiplication by nonzero scalars, 
because they are the irreducible factors of the $k$-root in the unique factorization domain $\Q[x_1, x_2, \ldots, x_n]$.
\end{rmk}

\begin{lemma}\label{lem:normform}
A $k$-root is positive if and only if it can be written in the form $$
\prod_{r = 1}^k (x_{i_{2r-1}} \pm x_{i_{2r}})
,$$ where the indices satisfy $i_{2r-1} < i_{2r}$ for all $i$.
\end{lemma}

\begin{proof} 
Let $\al$ be a $k$-root. We can reorder each factor of $\al$ to be of the form $(\pm x_i \pm x_j)$ where $i < j$, and commute the negative
signs on $x_i$ to the front to obtain $$
\al = \pm \prod_{r = 1}^k (x_{i_{2r-1}} \pm x_{i_{2r}})
,$$ where $i_{2r-1} < i_{2r}$ for all $i$. 
The lexicographically minimal monomial appearing in $\al$ is $$
x_{i_1} x_{i_3} \cdots x_{i_{2k-1}}
,$$ which occurs with coefficient $1$ if $\al$ is in the form given in the statement, and with coefficient $-1$ if $-\al$ is in the form given
in the statement. The conclusion now follows.
\end{proof}

\begin{defn}
We say that a positive $k$-root is in {\it normal form} if it is in the form given by Lemma \ref{lem:normform}. If $\al$ is a 
positive $k$-root in normal form, then we call a factor of $\al$ {\it antisymmetric} if it is of the form $(x_i - x_j)$ for $i < j$, 
and {\it symmetric} if it is of the form $(x_i + x_j)$. If $x_i$ does not appear in the factorization of $\al$, then we say that $i$ 
is an {\it unused index} of $\al$.
\end{defn}

Note that the normal form of a $k$-root is unique up to reordering the factors.

In the next result, $\chi^{(n-i,i)}$ refers to the character of the irreducible $\C S_n$-module corresponding to the two part
partition $(n-i, i)$ of $n$, as in \cite[\S4]{fulton13}. Part (ii) is well known.

\begin{lemma}\label{lem:span}
\begin{itemize}
\item[{\rm (i)}]{The set $\cnk$ is a spanning set for $\vnk$ over $\Q$.}
\item[{\rm (ii)}]{The character of $\vnk$ as a $\C S_n$-module is $\sum_{i = 0}^k \chi^{(n-i, i)}$.}
\end{itemize}
\end{lemma}

\begin{proof}
Let $x_I = x_{i_1}x_{i_2} \cdots x_{i_k}$ be an arbitrary monomial basis element. Because $k \leq n/2$, we can choose indices 
$j_1, j_2, \ldots, j_k$ so that the set $\{i_1, i_2, \ldots, i_k, j_1, j_2, \ldots, j_k\}$ has cardinality $2k$. By substituting 
$(x_{i_r}+x_{j_r})+(x_{i_r}-x_{j_r})$ for $2x_{i_r}$, the monomial $2^k x_I$ can be expressed as a sum of $2^k$ distinct
$k$-roots, and (i) follows.




Since $\vnk$ is induced from the trivial module of $S_k \times S_{n-k}$, its character corresponds to the product of the Schur
functions $s_{(n-k)}$ and $s_{(k)}$. The Pieri rule shows that we have $$
s_{(n-k)} s_{(k)} = \sum_{i = 0}^k s_{(n-i, i)}
,$$ which proves (ii).
\end{proof}

\begin{defn}\label{def:defects}
Let $\al \in \cnk$ be a positive $k$-root. We say that $\al$ has a {\it defect} if the normal form of $\al$
has any of the following features, where in each case we have $i < j < r < s$:
\begin{itemize}
    \item[(i)]{two factors of the form $(x_i \pm \xr)$ and $(x_j \pm \xs)$;}
    \item[(ii)]{a factor of the form $(x_i \pm \xs)$ and a symmetric factor of the form $(x_j + \xr)$;}
    \item[(iii)]{a factor of the form $(x_i \pm \xr)$, and an unused index $j$;}
    \item[(iv)]{a symmetric factor of the form $(x_i + x_j)$, and an unused index $r$.}
\end{itemize}
In case (i), we say that $\al$ has a {\it crossing}; in case (ii), we say that $(x_j + \xr)$ is a {\it nested symmetric factor}; in 
case (iii), we say that $j$ is a {\it nested unused index}; and in case (iv), we say that $(x_i + x_j)$ is an {\it obstructed
symmetric factor}.

We define $\bnk$ to be the set of positive $k$-roots with no defects.
\end{defn}

The next result is implicit in \cite[\S5]{gx3}, and will be very useful in the sequel.

\begin{prop}\label{prop:reduction}
Let $\al \in \cnk$ be a positive $k$-root, written in normal form, and suppose that $\al$ has a defect. Then $\al$ can be written
as a sum of positive $k$-roots that are strictly lower than $\al$ in the total order on $\vnk$, by replacing the factor(s) 
involved in the defect according to the following rules, where $i < j < r < s$ in all cases: 
\begin{align}
(x_i - \xr)(x_j - \xs) &\quad\longrightarrow\quad (x_i - x_j)(\xr - \xs) + (x_i - \xs)(x_j - \xr); \label{eq:1} \\
(x_i + \xr)(x_j - \xs) &\quad\longrightarrow\quad (x_i + x_j)(\xr - \xs) + (x_i + \xs)(x_j - \xr); \label{eq:2} \\ 
(x_i - \xr)(x_j + \xs) &\quad\longrightarrow\quad (x_i - x_j)(\xr + \xs) + (x_i + \xs)(x_j - \xr); \label{eq:3} \\
(x_i + \xr)(x_j + \xs) &\quad\longrightarrow\quad (x_i + x_j)(\xr + \xs) + (x_i - \xs)(x_j - \xr); \label{eq:4} \\
(x_i - \xs)(x_j + \xr) &\quad\longrightarrow\quad (x_i - x_j)(\xr + \xs) + (x_i + x_j)(\xr - \xs) + 
(x_i + \xs)(x_j - \xr); \label{eq:5} \\
(x_i + \xs)(x_j + \xr) &\quad\longrightarrow\quad (x_i - x_j)(\xr - \xs) + (x_i + x_j)(\xr + \xs) + 
(x_i - \xs)(x_j - \xr); \label{eq:6} \\
(x_i - \xr) &\quad\longrightarrow\quad (x_i - x_j) + (x_j - \xr); \label{eq:7} \\
(x_i + \xr) &\quad\longrightarrow\quad (x_i - x_j) + (x_j + \xr); \label{eq:8} \\
(x_i + x_j) &\quad\longrightarrow\quad (x_i - x_j) + (x_j - \xr) + (x_j + \xr). \label{eq:9}
\end{align}
The index $j$ in relations \ref{eq:7} and \ref{eq:8} and the index $r$ in relation \ref{eq:9} are assumed to be
unused indices of $\al$, and they may be chosen arbitrarily subject to the constraint that $i < j < r$.
\end{prop}

\begin{proof}
If $\al$ has a defect, then the normal form factorization of $\al$ must contain the factor(s) on the left hand side of one of
these nine identities, as follows. If $\al$ has a crossing, then one of the relations \ref{eq:1}, \ref{eq:2}, \ref{eq:3}, or 
\ref{eq:4} is applicable. If $\al$ has a nested symmetric factor, then one of the relations \ref{eq:5} or \ref{eq:6} is 
applicable. If $\al$ has a nested unused index, then one of the relations \ref{eq:7} or \ref{eq:8} is applicable. Finally, if
$\al$ has an obstructed symmetric factor, then relation \ref{eq:9} is applicable.

A routine verification shows that in each of the nine cases in the statement, the polynomial on the left hand side is equal to the
polynomial on the right hand side. By inspection, all of the factors appearing are of the correct type to appear in a normal
form factorization. It follows that making the substitutions indicated will express the normal form of the positive $k$-root $\al$
as a sum of other positive $k$-roots that are also in normal form.

It remains to show that if we have $\al = \sum_{p=1}^m \be_p$ as above, where the $\be_p$ are positive $k$-roots and $m > 1$, then 
we have $\be_p < \al$ for all $p$. This follows from the definition of the total order on $\vnk$, because we have $$
\al - \be_p = \sum_{q : 1 \leq q \leq m, q \ne p} \be_q
.$$ The right hand side is positive because it is a nontrivial sum of positive elements of $\vnk$, which means that
$\al - \be_p > 0$. By definition, this means that we have $\be_p < \al$, as required.
\end{proof}

\begin{rmk}\label{rmk:skein}
It may be convenient to visualize the relations in Proposition \ref{prop:reduction} as (singular) skein relations, in which the
factors of the form $(x_i - x_j)$ (respectively, $(x_i + x_j)$) are represented by an undecorated (respectively, decorated)
arc from $i$ to $j$. Figure \ref{fig:skein} shows the pictorial version of relation \ref{eq:6}, and Remark \ref{rmk:kl} gives
some more details on the relationship between $k$-roots and diagram algebras.   
\end{rmk}


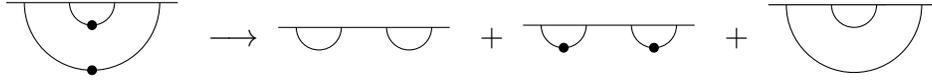
\begin{figure}[ht!]
    \centering
    \begin{tikzpicture}[anchorbase]      
        \draw (0,0) arc(-180:0:0.75*1.2) to (1.5*1.2,0);
        \draw (0.5*1.2,0) arc(-180:0:0.25*1.2) to (1*1.2,0);
        \draw (-0.2*1.2,0)--(1.7*1.2,0);
        \node[small node] at (0.75*1.2,-0.25*1.2) (a) {};
        \node[small node] at (0.75*1.2,-0.75*1.2) (b) {};
    \end{tikzpicture}
    \ \;$\longrightarrow$\;\
    \begin{tikzpicture}[anchorbase]
        \draw (0,0) arc(-180:0:0.25*1.2) to (0.5*1.2,0);
        \draw (1*1.2,0) arc(-180:0:0.25*1.2) to (1.5*1.2,0);
        \draw (-0.2*1.2,0)--(1.7*1.2,0);
    \end{tikzpicture}
    \ \;$+$\;\
        \begin{tikzpicture}[anchorbase]
        \draw (0,0) arc(-180:0:0.25*1.2) to (0.5*1.2,0);
        \draw (1*1.2,0) arc(-180:0:0.25*1.2) to (1.5*1.2,0);
        \draw (-0.2*1.2,0)--(1.7*1.2,0);
        \node[small node] at (0.25*1.2,-0.25*1.2) (a) {};
        \node[small node] at (1.25*1.2,-0.25*1.2) (b) {};
    \end{tikzpicture}
    \ \;$+$\;\
    \begin{tikzpicture}[anchorbase]
        \draw (0,0) arc(-180:0:0.75*1.2) to (1.5*1.2,0);
        \draw (0.5*1.2,0) arc(-180:0:0.25*1.2) to (1.1*1.2,0);
        \draw (-0.2*1.2,0)--(1.7*1.2,0);
    \end{tikzpicture}
    \caption{Relation \ref{eq:6} interpreted as a skein relation}
\label{fig:skein}
    \end{figure}

\begin{lemma}\label{lem:spanning}\ 
\begin{itemize}
    \item[{\rm (i)}]{Every positive $k$-root in $\cnk$ can be written as a linear combination of elements of $\bnk$ with
    nonnegative integer coefficients.}
    \item[{\rm (ii)}]{The set $\bnk$ is a spanning set for $\vnk$ over $\Q$, and we have $|\bnk| \geq \binom{n}{k}$.}
\end{itemize}  
\end{lemma}

\begin{proof}
To prove (i), suppose that $\al$ is a positive $k$-root. If we have $\al \in \bnk$, then we are done, so suppose that $\al$ has
a defect. We then apply Proposition \ref{prop:reduction} repeatedly to $\al$, applying the nine types of reduction in any order. 
This process must eventually terminate because the poset $\cnk$ is finite, and it will result in an expression for $\al$ as a 
positive integral linear combination of elements of $\bnk$.

The first assertion of (ii) follows from (i) and Lemma \ref{lem:span}, and the second assertion holds because $\vnk$ has
dimension $\binom{n}{k}.$
\end{proof}

\begin{rmk}\label{rmk:signature}
If $1 \leq i < j \leq n$ are integers, then the {\it signed transposition} $\overline{(i, j)}$ of the $2n$ symbols $\{\pm x_1, 
\pm x_2, \ldots, \pm x_n\}$ is the permutation that sends $\pm x_i$ to $\mp x_j$, and fixes $\pm x_r$ for $r \ne i, j$.
For each $k$-root $\al$, there is a natural way to assign a transposition $(i, j)$ to each antisymmetric factor $(x_i - x_j)$
and a signed transposition $\overline{(i, j)}$ to each symmetric factor $(x_i + x_j)$ 
to give a set $\{t_1, t_2, \ldots, t_k\}$ of distinct, mutually commuting signed and
unsigned transpositions with the property that $t_i(\al) = -\al$ for all $i$. Up to multiplying by nonzero scalars, the vector 
$\al \in \vnk$ can be characterized as the unique common eigenvector of the $\{t_1, t_2, \ldots, t_k\}$ with common 
eigenvalue $-1$.
\end{rmk}

\section{The canonical basis}\label{sec:canbas}

The main result of this section is Theorem \ref{thm:canbas}, which proves that the set $\bnk$ of $k$-roots without defects 
is a $\Q$-basis for $\vnk$, and that $\bnk$ can naturally be parametrized by a certain set of lattice words.
Recall that a {\it lattice word} is a sequence $a_1 a_2 \cdots a_n$ of positive integers with the property that for each 
positive integer $i$, each initial segment of the sequence contains at least as many occurrences of $i$ as of $i+1$. 

\begin{defn}\label{def:thelabal}
Let $\al \in \bnk$ be a positive $k$-root with no defects. We define the {\it label}, $\thelab(\al)$, of $\al$ to be the 
word of length $n$ in the alphabet $\{1, 2\}$ with the property that $\thelab(\al)_j = 2$ if and only if the normal form of
$\al$ has a factor of the form $(x_i - x_j)$ for some $i < j$. We define $\latnk$ to be the set of lattice words of length $n$ 
that have entries in the set $\{1, 2\}$, and at most $k$ occurrences of $2$.
\end{defn}

\begin{rmk}\label{rmk:tableaux}
The set $\latnk$ is in canonical bijection with the set of all standard Young tableaux with $n$ boxes having at most two rows
and at most $k$ boxes in the second row. The positions of the occurrences of $2$ in the lattice word correspond to the labels
of the boxes in the second row of the tableau.
\end{rmk}

Because each negative term $-x_j$ in the normal form is paired with a distinct term $x_i$ with $i < j$, the following result
follows immediately.

\begin{lemma}\label{lem:immed}
If $\al \in \bnk$ is a positive $k$-root with no defects, then we have $\thelab(\al) \in \latnk$.
\qed\end{lemma}

\begin{defn}
Let $\al \in \cnk$ be a positive $k$-root. We define the {\it height}, $h(\al)$, of $\al$ to be the number of symmetric
factors appearing in the normal form of $\al$. We define $\btnk{n}{k}{h}$ to be the subset of $\bnk$ consisting of $k$-roots
of height $h$.
\end{defn}

Note that if $\al \in \cnk$ then we always have $0 \leq h(\al) \leq k$.

\begin{lemma}\label{lem:deletion}\ 
\begin{itemize}
\item[{\rm (i)}]{There is a function $f : \btnk{n}{k}{h} \ra \btnk{n}{k-h}{0}$, where $f(\al)$ is defined to be the homogeneous
polynomial of degree $k-h$ obtained by removing the $h = h(\al)$ symmetric factors from the normal form of $\al$.}
\item[{\rm (ii)}]{The label $\thelab(\al)$ has $k-h(\al)$ occurrences of $2$, and satisfies $\thelab(\al) = \thelab(f(\al))$.}
\item[{\rm (iii)}]{The function $f$ is injective.}
\end{itemize}
\end{lemma}

\begin{proof}
It follows from the definition of normal form that removing $h$ factors from the normal form of a positive $k$-root will give
the normal form of a positive $(k-h)$-root. It remains to show that the resulting $(k-h)$-root has no defects.

Because $f(\al)$ has no symmetric factors by construction, it can have no defects of types (ii) or (iv) in 
Definition \ref{def:defects}. It is also immediate that the removal of factors cannot create new crossings, which means that 
$f(\al)$ has no defects of type (i) in Definition \ref{def:defects}. The only way $f(\al)$ can have a defect is if we are in
the situation of Definition \ref{def:defects} (iii), and because $f(\al)$ has no symmetric factors, we must be in the more
specific situation of relation \ref{eq:7} of Proposition \ref{prop:reduction}.

We may now assume that $f(\al)$ has a factor of the form $(x_i - \xr)$ and an unused index $j$ with $i < j < r$. Because $\al$
has no defects, the index $j$ must be involved in a symmetric factor $(x_j + x_m)$ of $\al$, for some $m \ne i, j, r$. 
We cannot have $m < i$ or $m > r$ because the factors $(x_i - \xr)$ and $(x_j + x_m)$ of $\al$ would create a crossing, and we 
cannot have $i < m < r$ because the factors $(x_i - \xr)$ and $(x_j + x_m)$ of $\al$ would create a nested symmetric factor. 
This completes the proof of (i).

Part (ii) follows from (i) and Definition \ref{def:thelabal}.

To prove (iii), we need to show that the $h$ symmetric factors of $\al$ are uniquely determined by $h$ and $f(\al)$. Because
$\al$ has no obstructed symmetric factors, it must be the case that the largest $2h$ unused indices of $f(\al)$ are precisely
the indices of the symmetric factors of $\al$. Let us denote these indices by $\{i_1, i_2, \ldots, i_{2h}\}$, where
$i_1 < i_2 < \cdots < i_{2h}$. Because $\al$ has no crossings and no nested symmetric factors, the symmetric factors of $\al$
must be $$
(x_{i_1} + x_{i_2}),
(x_{i_3} + x_{i_4}), \ldots,  
(x_{i_{2h-1}} + x_{i_{2h}})
.$$ This completes the proof of (iii).
\end{proof}

\begin{exa}
Consider the case $n = 12$, $k = 5$, $h = 2$. Let $\al$ be the element of 
$\btnk{12}{5}{2}$ given by $$
\al = (x_2 - x_3)(x_5 + x_{10})(x_6 - x_9)(x_7 - x_8)(x_{11} + x_{12})
,$$ so that $f(\al)$ is the element of $\btnk{12}{3}{0}$ given by $$
f(\al) = (x_2 - x_3)(x_6 - x_9)(x_7 - x_8)
.$$ Both $\al$ and $f(\al)$ have the label $112111122111$, where the $2$s appear in positions $3$, $8$, and $9$.
The largest $2h (= 4)$ indices not appearing in $f(\al)$ are $\{5, 10, 11, 12\}$, from which it follows that the symmetric
factors appearing in $\al$ are $(x_5 + x_{10})$ and $(x_{11} + x_{12})$.
\end{exa}

\begin{lemma}\label{lem:htzero}
Let $\al \in \btnk{n}{k}{0}$ be a $k$-root of height $0$ that has no defects, and let $p(\al)$ be the normal form of $\al$. 
For each index $j$ satisfying $\thelab(\al)_j = 2$, define $g(j)$ to be the unique index for which 
$(x_{g(j)} - x_j)$ is a factor in $p(\al)$. Then $i = g(j)$ is the largest index $i < j$ such that both $\thelab(\al)_i = 1$ and 
$p(\al)$ contains no factor of the form $(x_i - x_m)$ for any $m < j$.
\end{lemma}

\begin{rmk}
If one replaces the $1$s in $\thelab(\al)$ by open parentheses and the $2$s by close parentheses, then the map $g$ in the statement
locates the open parenthesis that matches a given close parenthesis. Lemma \ref{lem:immed} shows that it is always possible
to find a match.
\end{rmk}

\begin{proof}[Proof of Lemma \ref{lem:htzero}]
We know that there is at least one index $i < j$ such that both $\thelab(\al)_i = 1$ and $p(\al)$ contains no factor of the form
$(x_i - x_m)$ for any $m < j$, because $g(j)$ itself satisfies these conditions. Suppose for a contradiction that there exists such 
an $i$ for which $g(j) < i < j$.

If $i$ is an unused index in $\al$, then it is a nested unused index relative to the factor $(x_{g(j)} - x_j)$, which contradicts
the assumption that $\al \in \btnk{n}{k}{0} \subseteq \bnk$. Because $\thelab(\al)_i = 1$, the only other possibility is for
$x_i$ to be involved in a factor of the form $(x_i - x_m)$ with $g(j) < i < j < m$. In this case, the pair $(x_{g(j)} - x_j)$,
$(x_i - x_m)$ forms a crossing, which also contradicts the assumption $\al \in \btnk{n}{k}{0}$, completing the proof.
\end{proof}

\begin{lemma}\label{lem:leq}
Maintain the above notation.
\begin{itemize}
\item[{\rm (i)}]{For each $0 \leq h \leq k$, the restriction of $\thelab$ to $\btnk{n}{k}{h}$ is injective.}
\item[{\rm (ii)}]{The labelling function $\thelab : \bnk \ra \latnk$ is injective.}
\end{itemize}
\end{lemma}

\begin{proof}
Lemma \ref{lem:htzero} shows that if $h = 0$ then we can use induction on $j$ to reconstruct $\al$ from $\thelab(\al)$. This
proves (i) in the case $h = 0$. The general result of (i) now follows by combining the result for $h = 0$ with Lemma
\ref{lem:deletion}.

For (ii), observe that the set $\bnk$ is the disjoint union of the sets $\btnk{n}{k}{h}$ for $0 \leq h \leq k$. If 
$\al \in \btnk{n}{k}{h}$, then the number of occurrences of $2$ in $\thelab(\al)$ is $k - h$. It follows that the images
of the sets $\btnk{n}{k}{h}$ for $0 \leq h \leq k$ are pairwise disjoint, which completes the proof.
\end{proof}

\begin{theorem}\label{thm:canbas}
Let $n \leq 2$ and $0 \leq k \leq n/2$ be integers, and let $\bnk$ be the set of positive $k$-roots with no defects.
\begin{itemize}
\item[{\rm (i)}]{The labelling function $\lambda : \bnk \ra \latnk$ is a bijection.}
\item[{\rm (ii)}]{The set $\bnk$ is a basis for $\vnk$ over $\Q$.}
\item[{\rm (iii)}]{Every positive $k$-root is $\bnk$-positive, with integer coefficients.}
\item[{\rm (iv)}]{The elements $\bnk$ are the only positive $k$-roots that cannot be written as positive linear 
combinations of other positive $k$-roots.}
\end{itemize}
\end{theorem} 

\begin{proof}
Let $T(n, j)$ be the number of lattice words of length $n$ in the alphabet $\{1, 2\}$ where there are precisely $j$ occurrences
of $2$. It is known (see Sequence A008315 of \cite{oeis}) that $$
|T(n, j)| = \binom{n}{j} - \binom{n}{j-1}
,$$ where we interpret $\binom{n}{-1}$ to be zero.

Lemma \ref{lem:deletion} now implies that we have $|\btnk{n}{k}{h}| \leq T(n, k-h)$, and summing over $h$ gives $$
|\bnk| = \sum_{h = 0}^k |\btnk{n}{k}{h}| \leq \sum_{h=0}^k T(n, k-h) = \sum_{h=0}^k T(n, h) 
= \sum_{h=0}^k \left[ \binom{n}{h} - \binom{n}{h-1} \right] = \binom{n}{k}
.$$

Lemma \ref{lem:spanning} (ii) now implies that the inequality of the last paragraph is an equality, and that the injective
maps of Lemma \ref{lem:leq} are bijective, proving (i).

Part (ii) follows because we have $|\bnk| = \binom{n}{k} = \dim(\vnk)$, and part (iii) follows by combining (ii) 
with Lemma \ref{lem:spanning} (i).

If $\al$ is a positive $k$-root that is not an element of $\bnk$, then $\al$ has a defect, and it can be written as a positive 
integral linear combination of other positive $k$-roots by Proposition \ref{prop:reduction}. On the other hand, if $\al \in \bnk$ 
and $\al$ is a positive linear combination of other positive $k$-roots, it follows from (iii) that each of these positive 
$k$-roots must be a scalar multiple of $\al$. By Remark \ref{rmk:scalars}, each of the positive $k$-roots must equal $\al$, 
which is a contradiction and proves (iv).
\end{proof}

From now on, we will call $\bnk$ the {\it canonical basis} of $\vnk$.

\begin{rmk}
Parts (iii) and (iv) of Theorem \ref{thm:canbas} are familiar in the context of root systems. They show that the canonical basis
may be characterized purely in terms of the vector space ordering on $\vnk$, without relying on the concept of defects at all.
\end{rmk}

\begin{exa}
An example of the basis $\bnk$ that does not come from a root system is the case $n=4$, $k=2$. There are 12 positive and 12 negative
$2$-roots, and the elements of $\bcnk{4}{2}$ and their labels are as follows.

\begin{center}
\begin{tabular}{ |c|c| } 
 \hline
 Canonical basis element & Label \\
 \hline
$(x_1 + x_2)(x_3 + x_4)$ & $1111$ \\
$(x_1 + x_2)(x_3 - x_4)$ & $1112$ \\
$(x_1 + x_4)(x_2 - x_3)$ & $1121$ \\
$(x_1 - x_2)(x_3 + x_4)$ & $1211$ \\
$(x_1 - x_4)(x_2 - x_3)$ & $1122$ \\
$(x_1 - x_2)(x_3 - x_4)$ & $1212$ \\
 \hline
\end{tabular}
\end{center}
The six positive roots that are not basis elements come from the left hand sides of 
relations \ref{eq:1}--\ref{eq:6} of Proposition \ref{prop:reduction}, taking $i=1$, $j=2$, $r=3$ and $s=4$.
\end{exa}

\begin{rmk}\label{rmk:kl}
There are other constructions of the basis $\bnk$. One of these comes from the Kazhdan--Lusztig basis $\{C_w\}$ from 
\cite{kl79}, specifically, the basis of the module arising from the left cell containing the permutation 
$(1, 2)(3, 4)\cdots(k-1, k)$ in type $D_n$ with $q=1$. One can also construct $\bnk$ from a basis for the generalized 
Temperley--Lieb algebra of type $D_n$, after specializing $q$ to $1$ and twisting by sign. The latter basis may be 
defined in terms of monomials as in \cite[\S6.2]{fan97}, or in terms of diagrams as in \cite{green98}, and the 
definition of ``defect" in this paper is closely related to the diagrammatic rules in \cite{green98}.
When $k=n/2$, it is necessary in all these constructions to take the union of two cells: the one just described, and its 
image under the automorphism that sends $x_n$ to $-x_n$ and fixes $x_i$ for $i < n$.
One can find module isomorphisms between these various constructions by using the characterization of Remark 
\ref{rmk:signature}.

The $k$-root approach has a significant advantage over these other constructions, which is that Proposition 
\ref{prop:reduction} makes it easy (a) to work out the effect of applying an arbitrary (signed) permutation $w$ to a basis 
element $\al$ and then (b) to express the result as a linear combination of basis elements.
\end{rmk}

\section{Main results}\label{sec:main}

In Section \ref{sec:main}, we explore some applications of $k$-roots and the basis $\bnk$ in representation theory.
We first show how $\bnk$ naturally gives rise to a composition series of $\vnk$ as an $S_n$-module. We refer the reader
to Fulton and Harris \cite{fulton13} for background information on the character theory of the symmetric groups.

\begin{defn}\label{def:vtnk}
Let $\vtnk{n}{k}{t}$ be the $\Q$-linear span of all $k$-roots in $\bnk$ that have height at most $t$; that is, $$
\vtnk{n}{k}{t} := \Span \left( \bigsqcup_{h=0}^t \btnk{n}{k}{h} \right) 
.$$
\end{defn}

\begin{prop}\label{prop:compseries}
\begin{itemize}
\item[{\rm (i)}]{The subspaces $\vtnk{n}{k}{t}$ of Definition \ref{def:vtnk} are $\Q S_n$-submodules of $\vnk$.}
\item[{\rm (ii)}]{The chain $$
\vtnk{n}{k}{-1} := 0 
< \vtnk{n}{k}{0}
< \vtnk{n}{k}{1}
< \cdots 
< \vtnk{n}{k}{k} = \vnk
$$ is a composition series of $\vnk$ as a $\Q S_n$-module.}
\end{itemize}

\end{prop}

\begin{proof}
Let $\al \in \bnk$ be a canonical basis element of height $h$, and let $w \in S_n$ be a permutation. One of $\pm w(\al)$ is a
positive $k$-root of height $h$. A routine case by case check shows that the reduction rules in Proposition \ref{prop:reduction}
all express a positive $k$-root as a linear combination of positive $k$-roots of the same, or lower heights. It follows that
$w(\al)$ is a linear combination of canonical basis elements of height at most $h$, and this proves part (i).

Note that for each $j$ satisfying $0 \leq j \leq k$, there exists an element of $\latnk$ with precisely $j$ occurrences of $2$;
for example, the word $1^{n-j}2^j$. Each such word corresponds via Theorem \ref{thm:canbas} (i) to an element of $\bnk$ of 
height $k-j$, so there exist basis elements of all possible heights $h$ in the range $0 \leq h \leq k$.
It follows that each step in the chain in (ii) corresponds to a strict submodule, and thus that the series has
$k+1$ nontrivial quotients $\vtnk{n}{k}{h}/\vtnk{n}{k}{h-1}$.

Lemma \ref{lem:span} (ii) implies that $\vnk$ is the 
direct sum of $k+1$ irreducible $S_n$-submodules. Since this is the same as the number of nontrivial quotients in the series of
(ii), it follows both that $\vnk$ is a direct sum of $k+1$ irreducible submodules over $\Q$, and that the series in (ii) is
a composition series.
\end{proof}

The next result is useful for determining when a positive $k$-root stays positive after a permutation acts on it.

\begin{lemma}\label{lem:staypos}
Let $\al \in \cnk$ be a positive $k$-root, and let $w \in S_n$ be a permutation. If $w(\al)$ is negative, then
the normal form of $\al$ must contain a factor $(x_i - x_j)$ for which $w(i) > w(j)$.
\end{lemma}

\begin{proof}
If $w(\al)$ is negative, but there is no factor in $\al$ of the form $(x_i - x_j)$ satisfying $w(i) > w(j)$, then each 
factor in the normal form of $\al$ is sent by $w$ to 
another factor in normal form. It follows from Lemma \ref{lem:normform} that $w(\al)$ is positive, which is a contradiction.
\end{proof}

Although we know that the irreducible components of $\vnk$ have characters $\chi^{(n-i, i)}$, we will need to be able to match these
to the composition factors of the series in Proposition \ref{prop:compseries}. The next result helps with this.

\begin{lemma}\label{lem:long}
Let $A$ be a subset of $\{1, 2, \ldots, n\}$ of cardinality $a$, and let $S_A$ be the full symmetric group on $A$ considered as a subgroup
of $S_n$ of order $a!$. Define $$
x_A := \sum_{w \in S_A} w
.$$ 
\begin{itemize}
\item[{\rm (i)}]{If $V_i$ is an irreducible $\C S_n$-module with character $\chi^{(n-i, i)}$ for some $i \leq n/2$, 
then we have $x_A.V_i \ne 0$ if and only if $a \leq n-i$.}
\item[{\rm (ii)}]{If $1 \leq j \leq k$, then the $k$-root $$
\be_j = \prod_{i = 1}^j (x_i - x_{k+i}) \prod_{i=j+1}^k (x_i + x_{k+i})
$$ lies in $\vtnk{n}{k}{k-j}$.} 
\item[{\rm (iii)}]{If $\vtnk{n}{k}{t}$ is as in Proposition \ref{prop:compseries}, then we have
$x_A. \vtnk{n}{k}{t} \ne 0$ if and only if $a \leq n-k+t$.}
\item[{\rm (iv)}]{The composition factor $\vtnk{n}{k}{t}/\vtnk{n}{k}{t-1}$ has character $\chi^{(n-k+t,k-t)}$, and $\vtnk{n}{k}{t}$
has character $$
\sum_{i = k-t}^k \chi^{(n-i, i)}
.$$}
\end{itemize}  
\end{lemma}

\begin{proof}
  If we define $B = \{1, 2, \ldots, a\}$, then we have $x_A = g x_B g^{-1}$ for
  some $g \in S_n$, which implies that $x_A$ and $x_B$ annihilate the same
  modules. It
is therefore enough to consider the case where $A = \{1, 2, \ldots, a\}$.

The condition that $x_A.V_i \ne 0$ is equivalent to the condition that $V_i$, when regarded as an $S_A$-module, contains a copy of the
trivial representation, that is, that $$
\langle 1, V_i \downarrow^{S_n}_{S_A} \rangle \ne 0
$$ in the usual inner product on characters. By Frobenius reciprocity, this is equivalent to $$
\langle 1 \uparrow_{S_A}^{S_n}, V_i \rangle \ne 0
.$$ Since $S_A$ is a Young subgroup of $S_n$ of type $S_a \times S_1 \times \cdots \times S_1$, where there are $n-a$ copies of $S_1$,
it follows that the character of $1 \uparrow_{S_A}^{S_n}$ corresponds to the product of Schur functions $$
s_{(a)} s_{(1)} \cdots s_{(1)}
,$$ where again there are $n-a$ copies of $s_{(1)}$. This corresponds to adding $n-a$ boxes, one at a time, to the partition $(a)$. This
will result in at least one copy of $s_{(n-i, i)}$ if and only if we have $n - a \geq i$; otherwise, there are not enough single boxes to
fill the second row. This proves (i).

Let the $k$-root $\be_j$ be as in the statement of (ii). 
Observe that $\be_j$ is a $k$-root of height $k-j$, and it is in the same
$S_n$-orbit as any element of $\btnk{n}{k}{k-j}$, for example, the
element of $\bnk$ whose label is $1^{n-j}2^j$. Since $\be_j$ is
in the same $S_n$-orbit as an element of $\vtnk{n}{k}{k-j}$, it follows that
$\be_j \in \vtnk{n}{k}{k-j}$, proving (ii).

To prove (iii), note that any basis element $\al$ of $\vtnk{n}{k}{t}$ has at least $k-t$ antisymmetric factors. If we have 
$a > n-(k-t)$, then 
it is inevitable that at least one of the antisymmetric factors of $\al$ is of the form $(x_i - x_j)$ where both of $i$ and $j$ lie 
in $A$. It follows that $\al$ is annihilated by $1 + w$, where $w$ is the transposition $(i, j)$. By summing over a set of
left coset representatives of $\langle w \rangle$ in $S_A$, we can factorize $x_A$ as $x'(1+w)$, from which it follows that 
$x_A$ annihilates $\al$. Since $\al$ was arbitrary, we deduce that $x_A$ annihilates $\vtnk{n}{k}{t}$ if $a > n-k+t$.

It remains to show that if $a \leq n-k+t$, then $x_A$ does not annihilate $\vtnk{n}{k}{t}$. Let $\be = \be_{k-t}$ be the
$k$-root defined in (ii).
By construction, no antisymmetric factor of $\be$ has both endpoints in the set $A$. Lemma \ref{lem:staypos} now implies that
any $w \in S_A$ has the property that $w(\be)$ is a positive $k$-root. It follows that $x_A(\be)$ is a nontrivial sum of 
positive $k$-roots, and Theorem \ref{thm:canbas} (iii) shows that $x_A(\be)$ is a nontrivial sum of canonical basis elements.
In particular, $x_A$ does not annihilate $\be$, and therefore $x_A$ does not annihilate $\vtnk{n}{k}{t}$, proving (iii).

If we set $A = \{1, 2, \ldots, n-k+t\}$ and $B = \{1, 2, \ldots, n-k+t+1\}$, then (iii) implies that $x_A$ annihilates 
$\vtnk{n}{k}{t-1}$, but not $\vtnk{n}{k}{t}$, and that $x_B$ annihilates $\vtnk{n}{k}{t}$. The character of
$\vtnk{n}{k}{t}/\vtnk{n}{k}{t-1}$ is therefore the character of the form $\chi^{(n-i, i)}$ that is annihilated by $x_B$
but not by $x_A$. By (i), we find the solution is to take $i = k-t$, which proves the first assertion of (iv). The second assertion
of (iv) follows by summing over all the composition factors of $\vtnk{n}{k}{t}$.
\end{proof}

Recall (for example, from sections 2 and 3 of \cite{diaconis87} or the proof of \cite[Corollary 4.6.4 (ii)]{ceccherini08}) that 
the spherical functions $\pnkj$ are characterized by the following properties:
\begin{itemize}
    \item[(i)]{$\pnkj$ lies in the irreducible summand of $L(X)$ with character $\chi^{(n-j, j)}$;}
    \item[(ii)]{$\pnkj$ is fixed pointwise by the subgroup $K = S_k \times S_{n-k}$;}
    \item[(iii)]{$\pnkj$ takes the value $1$ at the identity coset; in other words, the coefficient of $x_1 x_2 \cdots x_k$ in
    $\pnkj$ is $1$.}
\end{itemize}
We are now ready to give a construction of these spherical functions in terms of $k$-roots.

\begin{theorem}\label{thm:main}
Let $n \geq 2$ and $0 \leq k \leq n/2$, and $0 \leq j \leq k$ be integers. Let $A = \{1, 2, \ldots, k\}$ and let 
$B = \{j+1, j+2, \ldots, n\}$.
\begin{itemize}
\item[{\rm (i)}]{As a homogeneous polynomial in $x_1, x_2, \ldots, x_n$, the $j$-th spherical function $\pnkj$
of the Gelfand pair $(S_n, S_k \times S_{n-k})$ is given by $$
\frac{(n-2k)!}{k!(n-k)!2^{k-j}(k-j)!} \left( \sum_{v \in S_A} v \right) \left( \sum_{w \in S_B} w \right) \cdot \be_j
,$$ where $\be_j$ is the $k$-root defined by $$
\be_j = \prod_{i = 1}^j (x_i - x_{k+i}) \prod_{i=j+1}^k (x_i + x_{k+i}) 
.$$}
\item[{\rm (ii)}]{The function $\pnkj$ is $\bnk$-positive, and its coefficients 
are nonnegative integer multiples of $1/N$, where $N$ is the integer $$
(k)_j(n-k)_k = \frac{k!(n-k)!}{(k-j)!(n-2k)!}
.$$}
\end{itemize}
\end{theorem}

\begin{proof}
Let $\psnkj$ be the polynomial given in the formula; we will show that $\psnkj$ is equal to the $j$-th spherical function,
$\pnkj$. Any permutation in $S_A$ will fix $\psnkj$, 
because we have already symmetrized over $S_A$. Any permutation of $\{k+1, k+2, \ldots, n\}$ will commute with each element of
$S_A$, and be equal to an element of $S_B$, so these too will fix $\psnkj$. It follows that $\psnkj$ is fixed by the subgroup
$K = S_k \times S_{n-k}$.

We next prove that $\psnkj$ lies in the unique irreducible submodule $V_j$ of $\vnk$ with character $\chi^{(n-j, j)}$. 
The $k$-root $\be_j$ lies in $\vtnk{n}{k}{k-j}$ by Lemma \ref{lem:long} (ii), and $\vtnk{n}{k}{k-j}$ 
has character $\sum_{i = j}^k \chi^{(n-i, i)}$ by Lemma \ref{lem:long} (iv).
It is enough to show that $\sum_{w \in S_B} w \cdot \be_j$ lies in $V_j$.
Since $B$ has cardinality $n-j$, it follows from Lemma \ref{lem:long} (i) that 
$\sum_{w \in S_B} w$ will annihilate every submodule whose character is in the set $$
\{\chi^{(n-i, i)} : j \leq i \leq k\}
$$ except the one with character $\chi^{(n-j, j)}$. It follows that
$\sum_{w \in S_B} w \cdot \be_j$ lies in $V_j$, as required.

In order to complete the proof of (i), it remains to show that $x_1 x_2 \ldots x_k$ appears in $\psnkj$ with coefficient $1$.
Observe that the polynomial $\prod_{i=j+1}^k (x_i + x_{k+i})$ is stabilized by a subgroup $U \leq S_B$, generated
by transpositions $(i, k+i)$ that fix each factor, together with permutations of the $k-j$ factors. It follows that $U$ has order
$|U|=2^{k-j}(k-j)!$. By summing over a set of left coset representatives, $X_B$, for the left cosets $S_B/U$ of $U$ in $S_B$, we 
can obtain an expression for $\psnkj$ that is equivalent to the one in the statement but has fewer terms, as follows: $$
\psnkj = \frac{(n-2k)!}{k!(n-k)!} \left( \sum_{v \in S_A} v \right) \left( \sum_{w \in X_B} w \right) \cdot \be_j
.$$

Note that every antisymmetric factor of $\be_j$ contains precisely one index from the set $\{1, 2, \ldots, j\}$, 
and these indices are fixed pointwise by every element in $S_B$. Lemma \ref{lem:staypos} shows that
the $k$-roots $\{w \cdot \be_j : w \in X_B\}$ are all positive, and because we are summing over cosets of the stabilizer, 
each element of $X_B$ gives a different positive $k$-root $w \cdot \be_j$. It is convenient to separate the $k$-roots 
$w \cdot \be_j$ into three mutually exclusive types.

\begin{itemize}
\item[Type 1:]{positive $k$-roots containing at least one antisymmetric factor $(x_p - x_q)$ where $1 \leq p < q \leq k$;} 
\item[Type 2:]{positive $k$-roots that are not of type 1, but that contain least one symmetric factor $(x_p + x_q)$ where 
either $1 \leq p < q \leq k$ or $k+1 \leq p < q \leq n$;} 
\item[Type 3:]{positive $k$-roots where each factor contains precisely one $x_p$ where $1 \leq p \leq k$.} 
\end{itemize}

The $k$-roots of type 1 are annihilated by elements of the form $1 + w$ where $w$ is the transposition $(p, q)$. As in the
proof of Lemma \ref{lem:long} (iii), it follows that $k$-roots of type $1$ are annihilated by
$\sum_{v \in S_A} v$, and thus that they make no net contribution to the sum. These terms may be ignored from now on.

The $k$-roots of types 2 and 3 contain no antisymmetric factors $(x_p - x_q)$ with $1 \leq p < q \leq k$. If $\al$ is a $k$-root
of type 2 or 3 and $v \in S_A$, it follows by Lemma \ref{lem:staypos} that $v \cdot \al$ is also a positive $k$-root.

If $\al$ has type 2, then the $k$-root $v \cdot \al$ will have a factor $(x_{v(p)} + x_{v(q)})$ with two indices in the range 
$1 \leq v(p), v(q) \leq k$ or in the range $k +1 \leq v(p), v(q) \leq n$, and this means that 
$x_1 x_2 \cdots x_k$ appears in $v \cdot \al$ with coefficient zero.

A $k$-root $\al$ is of type $3$ if and only if it has the form $$
\al = \prod_{i = 1}^j (x_i - x_{\iota(i)}) \prod_{i=j+1}^k (x_i + x_{\iota(i)}) 
$$ for some injective function $\iota : \{1, 2, \ldots, k\} \ra \{k+1, k+2, \ldots, n\}$. The monomial $x_1 x_2 \cdots x_k$ 
appears in each such $k$-root with coefficient $1$, and the number of $k$-roots of type $3$ is the same as the number of
functions $\iota$, which is $(n-k)_k = (n-k)!/(n-2k)!$. The action of a permutation $v \in S_A$ leaves invariant the coefficient
of $x_1 x_2 \cdots x_k$, which implies that the coefficient of $x_1 x_2 \cdots x_k$ in 
$\left( \sum_{v \in S_A} v \right) \left( \sum_{w \in X_B} w \right) \cdot \be_j$ is $k!(n-k)!/(n-2k)!$, proving (i).

The above argument has also shown that $\be'_j := \left( \sum_{v \in S_A} v \right) \left( \sum_{w \in X_B} w \right) \cdot \be_j$ 
is a sum of positive $k$-roots: the terms 
of type 1 all cancel, and the terms of types 2 and 3 lead to sums of positive $k$-roots. Theorem \ref{thm:canbas} (iii) implies 
that $\be'_j$ is a linear combination of canonical basis elements with nonnegative integer coefficients. Now let $C = A \cap B$,
so that $|C|=k-j$, and let $X_A$ be a set of left coset representatives of $S_C$ in $S_A$. The left $S_B$-invariance of
$\sum_{w \in X_B} w \cdot \be_j$ then implies that $$
\be'_j 
= \left( \sum_{v \in X_A} v \right) \left( \sum_{u \in S_C} u \left( \sum_{w \in X_B} w .\be_j \right) \right)
= (k-j)! \left( \sum_{v \in X_A} v \right) \left( \sum_{w \in X_B} w \right) \cdot \be_j
.$$ It follows that the coefficients of the canonical basis elements in $\be'_j$ are all integer multiples of $(k-j)!$,
and dividing by the factor of $k!(n-k)!/(n-2k)!$ from the previous paragraph then proves (ii).
\end{proof}

\begin{rmk}
The bound on the denominator given in Theorem \ref{thm:main} (ii) is sharp in some cases, such as the case $n=4$ and $k=j=1$ in
which the denominator in the theorem is the best possible denominator of $1/3$.
\end{rmk}

Finally, we consider the problem of expressing elements of $\mnk$ as linear combinations of the canonical basis $\bnk$.
The basis $\bnk$ only has one element that is a positive linear combination of the natural basis of monomials, namely the basis
element whose label is $1^n$. However, it often happens that a squarefree monomial can be written as a positive
linear combination of the $\bnk$. The following definition is helpful for understanding this.

\begin{defn}
Let $x_I = x_{i_1} x_{i_2} \cdots x_{i_k}$ be a squarefree monomial of degree $k$ in the indeterminates $x_1, \ldots x_n$. Define
the {\it label}, $\mu(x_I)$ of $x_I$ to be the sequence of length $n$ in the alphabet $\{1, 2\}$ with the property that $\mu(x_I)_j = 2$ if
and only if $x_j$ appears in $x_I$. We say that $\mu(x_I)$ is a {\it reverse lattice word} if every terminal segment of $\mu(x_I)$
contains at least as many $1$s as $2$s.
\end{defn}

\begin{theorem}\label{thm:revlat}
Let $n \geq 2$ and $0 \leq k \leq n/2$ be integers, and let $x_I = x_{i_1} x_{i_2} \cdots x_{i_k}$ be a squarefree monomial of 
degree $k$ in the indeterminates $x_1, \ldots x_n$. If the label $\mu(x_I)$ is a reverse lattice word, then $x_I$ is $\bnk$-positive,
with coefficients that are nonnegative integer multiples of $1/2^k$.
\end{theorem}

\begin{proof}
Suppose that $x_I$ satisfies the hypotheses in the statement. It is enough to prove that $2^k x_I$ is $\bnk$-positive with integer
coefficients.

The hypothesis that $\mu(x_I)$ is a reverse lattice word means that there is an injective function 
$f : \{i_1, i_2, \ldots, i_k\} \rightarrow \{1, 2, \ldots, n\}$ with the properties that for all $1 \leq r \leq k$, both
(a) $f(i_r) > i_r$ and (b) $\mu(x_I)$ has a $1$ at position $f(i_r)$. By making the substitutions $$
2x_{i_r} \ra (x_{i_r} - x_{f(i_r)}) + (x_{i_r} + x_{f(i_r)})
,$$ we can express $2^k x_I = (2x_{i_1})(2x_{i_2}) \cdots (2x_{i_k})$ as an sum of $2^k$ positive $k$-roots in normal form. 
The result now follows from Theorem \ref{thm:canbas} (iii).
\end{proof}

\begin{rmk}
The number of reverse lattice words with $k$ occurrences of $2$ is $\binom{n}{k} - \binom{n}{k-1}$. It therefore follows that all but
at most $\binom{n}{k-1}$ elements of the monomial basis are $\bnk$-positive. 
If $k$ is small compared to $n$, then the hypotheses of Theorem \ref{thm:revlat} will usually be satisfied, but if $k$ is close to
$n/2$, the hypotheses will rarely be satisfied. The monomial $x_1 x_2 \cdots x_k$, whose label is $2^k 1^{n-k}$, will always satisfy 
the hypotheses.

It would be interesting to know whether the necessary condition in the theorem is also sufficient. This is the case when $k=1$, where
$x_n$ is the only monomial that is not $\bcnk{n}{1}$-positive.
\end{rmk}

\section{Concluding remarks}

\subsection{Other Gelfand pairs}\label{subsec:ogp}

A natural question is whether the results of this paper have analogues for other Gelfand pairs. This happens, for example,
in the case of the Gelfand pair $(S_{2n}, S_n \wr (\Z/2\Z))$, which corresponds to the action of
$S_{2n}$ on the size $n$ subsets of $\{1, 2, \ldots, 2n\}$, where each subset is identified with its complement. 
In this case, the coset space of the Gelfand pair has a canonical basis induced by the 
elements of $\bcnk{2n}{n}$ that have an even number of asymmetric factors.

An example of a Gelfand pair that has a simpler treatment
than the one in this paper is $((\Z/2\Z) \wr S_n, S_n)$. In this case, the $2^n$ cosets correspond to the weights
of the spin representation of a simple Lie algebra of type $B_n$ \cite[Proposition 6.4.5]{green13}. There is a well-known action 
of the Weyl group $W(B_n) \cong (\Z/2\Z) \wr S_n$ on
$2n$ symbols $\{ 1, \overline{1}, 2, \overline{2}, \ldots, n, \overline{n}\}$ \cite[Example 1.4.5]{green13}. This induces an 
action by signed permutations on the span of the $2^n$ linearly independent polynomials in the $2n$ commuting indeterminates
$x_1, x_{\overline{1}}, \ldots, x_n, x_{\overline{n}}$ of the form $$
(x_1 \pm x_{\overline{1}})(x_2 \pm x_{\overline{2}}) \cdots (x_n \pm x_{\overline{n}})
,$$ where the signs are chosen independently. This is a basis for the permutation module on the cosets that is compatible with
the direct sum decomposition into $W(B_n)$-irreducibles, and the corresponding spherical functions are given by taking the average of
each $S_n$-orbit of basis elements.

The results of this paper can also be thought of in terms of averaging operators.  
It follows from Theorem \ref{thm:main} that when the spherical functions of $(S_n, S_k \times S_{n-k})$ are written as linear 
combinations of canonical basis elements, the denominators of the coefficients
divide the order of $K$, namely $k!(n-k)!$. Using this, one can replace the expression in Theorem \ref{thm:main} (i) by an averaging 
operator over $K$ acting on a sum with far fewer terms. This suggests that there may be continuous versions of these results in which 
the averaging operator is replaced by a suitable integral.

\subsection{Sign-coherence}

Because the set $\cnk$ of $k$-roots is permuted by the action of any permutation $w$, it follows from Theorem
\ref{thm:canbas} (iii) that the matrix $\rho(w)$ representing $w$ with respect to $\bnk$ is an integer valued column sign-coherent 
matrix. The property of {\it column sign-coherence} comes from the theory of cluster algebras 
(\cite[Definition 2.2 (i)]{cao19}, \cite[Definition 6.12]{fomin07}, \cite[\S5]{gx3}), 
and means that any two nonzero entries in the same column of $\rho(w)$ have the same sign. Each simple $S_n$-module 
$\vtnk{n}{k}{t}/\vtnk{n}{k}{t-1}$ inherits a basis from $\bnk$ that also has the sign-coherence property.
This sign-coherence property is remarkable because it fails easily
for irreducible $S_n$-modules corresponding to partitions with more than two rows; for example, the irreducible module for 
$S_4$ with character $\chi^{(2,1,1)}$ contains a counterexample. The monomial basis for $\vnk$ and the basis mentioned in 
\ref{subsec:ogp} both have the sign-coherence property, but in the trivial sense that the matrices representing group 
elements have only one nonzero entry per column.

\subsection{Differential operators}

Some of the results of this paper say something about the differential operators $d : \vcnk{n}{k} \rightarrow
\vcnk{n}{k-1}$ given by $d = \sum_{i=1}^n \partial/\partial x_i$. It follows from the definitions that $d$ sends
positive $k$-roots to linear combinations of positive $k$-roots with positive even integer coefficients. 
Theorem \ref{thm:canbas} (iii) then implies that the entries of the matrix of $d$ relative to $\bcnk{n}{k}$ and 
$\bcnk{n}{k-1}$ are positive even integers. The submodules $\vtnk{n}{k}{t}$ of Definition \ref{def:vtnk} can be 
simply characterized as the kernels of the composite operators  $d^{t+1}$, as in \cite[Theorem 6.1.6 (v)]{ceccherini08}.

\subsection{Categorification}

The appearance of $\bnk$-positivity in various contexts in this paper raises the question of whether the positive integers and rational 
numbers that arise have combinatorial interpretations. A related question is whether $k$-roots can be categorified, and the connection 
with Kazhdan--Lusztig bases mentioned in Remark \ref{rmk:kl} is an additional hint that this may be possible. It seems likely that the
reduction rules in Proposition \ref{prop:reduction} would play an important role in any such categorification.

\section*{Acknowledgements}
I am grateful to Nathan Lindzey and Nat Thiem for some helpful conversations,
and to Tianyuan Xu for making many helpful comments and suggestions on an
earlier version of this paper. I also thank the referees for their corrections
and feedback.

\section*{Data availability statement}

Data sharing not applicable to this article as no datasets were generated or analysed during the current study.

\section*{Conflict of interest statement}

On behalf of all authors, the corresponding author states that there is no conflict of interest. 

\bibliographystyle{plain}
\bibliography{ppsf.bib}

\end{document}